\documentclass[11pt]{amsart}

\usepackage{CJKutf8}

\usepackage{amsmath}
\usepackage{amsthm} 
\usepackage{graphicx}
\usepackage{color}
\usepackage{scalerel}
\usepackage[dvipsnames]{xcolor}
\usepackage{stmaryrd}
\usepackage{soul}
\usepackage{tikz}
\usetikzlibrary{decorations.pathreplacing}

\usepackage{enumerate}

\usepackage{amsfonts}
\usepackage{amssymb}
\usepackage[hidelinks]{hyperref}
\usepackage[capitalize]{cleveref}

\usepackage{nicefrac}
\usepackage{enumitem}

\usepackage{tipa}
\usepackage[normalem]{ulem}

\usepackage{mathtools}

\usepackage{lmodern}

\usepackage[style=authoryear,labelnumber=true,sorting=none]{biblatex} 
\DeclareFieldFormat{labelnumberwidth}{\mkbibbrackets{#1}}

\defbibenvironment{bibliography}
  {\list
     {\printtext[labelnumberwidth]{%
        \printfield{labelnumber}%
        \printfield{extraalpha}}}
     {\setlength{\labelwidth}{\labelnumberwidth}%
      \setlength{\leftmargin}{\labelwidth}%
      \setlength{\labelsep}{\biblabelsep}%
      \addtolength{\leftmargin}{\labelsep}%
      \setlength{\itemsep}{\bibitemsep}%
      \setlength{\parsep}{\bibparsep}}%
      }
  {\endlist}
  {\item}

\DeclareMathOperator*{\forkindep}{\raise0.2ex\hbox{\ooalign{\hidewidth$\vert$\hidewidth\cr\raise-0.9ex\hbox{$\smile$}}}}

\DeclareTextCommand{\DZ}{OT2}{D2}

\newtheorem*{claim-star}{Claim}
\newtheorem*{theorem-non}{Theorem}
\newtheorem{theorem}{Theorem}[section] 
\newtheorem{lemma}[theorem]{Lemma}

\newtheorem{prop-def}[theorem]{Proposition-Definition}
\newtheorem{corollary}[theorem]{Corollary}

\newtheorem{fact-eh}[theorem]{Fact(?)}

\newtheorem{conjecture}[theorem]{Conjecture}

\newtheorem{question}[theorem]{Question}
\newtheorem{proposition}[theorem]{Proposition}
\newtheorem{proposition-eh}[theorem]{Proposition(?)}
\newtheorem*{theorem-star}{Theorem}
\newtheorem*{conjecture-star}{Conjecture}
\newtheorem*{lemma-star}{Lemma}

\theoremstyle{definition}
\newtheorem{definition}[theorem]{Definition}
\newtheorem{example}[theorem]{Example}

\newtheorem{remark}[theorem]{Remark}

\theoremstyle{remark}
\newtheorem{claim}[theorem]{Claim}

\newenvironment{claimproof}[1][\proofname]
               {
                 \proof[#1]
                 
               }
               {
                 \endproof
               }

\allowdisplaybreaks 

\title{Enumeration of Finite Distance Monoids}

\author{Yunjie Luo} 
\address{\'Ecole Polytechnique}
\email{yunjie.luo@polytechnique.edu}

\author{Jie Sheng}
\address{Sorbonne Universit\'e}
\email{jie.sheng@etu.sorbonne-universite.fr}

\thanks{Partially supported by NSFC grants 12426507 \& 12501001}
\begin{document}

\maketitle

\begin{abstract}
    Building on the work of Gabriel Conant [\cite{Conant2015}], we investigate the enumeration problems of finite distance monoids by applying the decomposition of Archimedean classes and studying their internal arithmetic progressions. Specifically, we first determine the exact value of $DM(n,2)$, which denotes the number of distance monoids on $n$ non-zero elements with Archimedean complexity $2$. This computation allows us to resolve a conjecture of Conant, establishing that the total number $DM(n)$ of distance monoids grows at least exponentially in $n$. Furthermore, we study the asymptotic behavior of $DM(n,n-k)$ for fixed $k$, proving that $DM(n,n-k) = O(n^k)$ and providing an exact formula for $DM(n,n-2)$. 
\end{abstract}

\section{Introduction}
A distance monoid is a commutative monoid equipped with a compatible total order which appears under various names in the literature, such as a totally ordered commutative monoid [\cite{Evans2001}], a positively ordered monoid (P.O.M.) [\cite{wehrung:hal-00004710}]. These objects naturally arise in different areas of mathematics. One such instance is in the thesis of Gabriel Conant [\cite{Conant2015}], where some surprising connections between finite distance monoids, additive and algebraic combinatorics, and model theory occur. Conant formulated several conjectures concerning the enumeration and structure of finite distance monoids in $\mathbb{R}$. Results related to the latter was obtained by Veljko Toljić [\cite{toljic2023conant}]. Finite distance monoids also play an important role in computer science and related applied fields. For instance, they appear as discrete triangular norms (t-norms) in fuzzy logic and aggregation operator theory [\cite{DeBaets2003}; \cite{klement2013triangular}].

Our work is motivated by the enumeration of finite distance monoids. We can regard distance monoids as a special type of finite distance magmas, i.e., commutative magmas equipped with a compatible total order. Conant [\cite{Conant2015}] showed that there is a bijection between finite distance magmas with $n$ non-zero elements and Magog triangles of order $n$. Since Propp [\cite{propp2001many}] proved that the number of magog triangles of order $n$ is the Robbins number, the enumeration of finite distance magmas is resolved. However, enumerating finite distance monoids presents a greater challenge, primarily due to the difficulty of directly verifying associativity. To study this problem, Conant [\cite{Conant2015}] introduced the concept of Archimedean complexity, which describes the ``complexity" of the addition. The problem thus reduces to determining the number of finite distance monoids with distinct Archimedean complexities. Let $DM(n,k)$ denote the number of distance monoids with $n$ non-zero elements and Archimedean complexity $k$. Conant demonstrated $DM(n,1)=DM(n,n)=1$ and $DM(n,n-1)=2n-2$. 
    
    In this paper, we make significant progress on the enumeration of finite distance monoids. Our first main result is the exact enumeration of $DM(n,2)$.
    \begin{theorem}\label{DM(n,2)} For all $n\geq 2$,
            \[DM(n,2)=\sum_{k=1}^{n-1}\sum\limits_{\substack{n_{1}+\cdots+n_{k}=n \\ n_{i}\in \mathbb{N}_{>0}}}\prod\limits_{j=1}^k j^{n_j-1}.\]
    \end{theorem}
    This formula reveals a connection between $DM(n,2)$ and Bell numbers, which count the number of partitions of a set. Using this connection, we resolve  [\cite{Conant2015}, Conjecture 5.1.8(b)], which states that the total number $DM(n)$ of distance monoids on $n$ non-zero elements grows at least as fast as any exponential function, i.e., $DM(n)=\Omega(b^n)$ for every $b>0$, where $f(n)=\Omega(g(n))$ means that there are $c\in \mathbb{R}_{>0}$ and $n_0\in \mathbb{N}_{>0}$ such that $f(n)\geq cg(n)$ for all $n>n_0$.
    
    Our second main result concerns the asymptotic behavior of $DM(n,n-k)$ for fixed $k$.
    
    \begin{theorem}\label{DM(n,n-k)}
            For any fixed $k\in \mathbb{N}$, $DM(n,n-k)=O(n^k).$
    \end{theorem}
    
    This theorem is proved by analyzing the arithmetic progressions within the finite distance monoid. We also provide an exact formula for $DM(n,n-2)$.

    The paper is organized as follows. In Section~\ref{sec:preliminaries}, we review the necessary background on distance monoids and Archimedean classes. Section~\ref{sec:enumeration} is devoted to the proof of Theorem \ref{DM(n,2)} and its connection to Bell numbers. In Section~\ref{sec:arithmetic}, we examine the relationship between arithmetic progressions and Archimedean complexity, leading to the proof of Theorem \ref{DM(n,n-k)}. Finally, in Section~\ref{sec:asymptotic}, we analyze $DM(n,n-k)$ for large $n$ and provide the exact formula for $DM(n,n-2)$.

\section*{Acknowledgement}
    This research was conducted during the PACE (PKU Algebraic Combinatorics Experience) summer REU program at Peking University under the supervision of Kyle Gannon. The authors are deeply grateful to the Beijing International Center for Mathematical Research (BICMR) and Professor Yibo Gao for providing this valuable opportunity. We would also like to express our sincere appreciation to our advisor, Kyle Gannon, for selecting such a suitable and engaging research topic and for actively engaging in discussions with us on related problems. Without his guidance and support, this paper could not have been completed.

\section{Notation and Preliminaries}\label{sec:preliminaries}

    In this section, we establish the notation and basic definitions that will be used throughout the paper. First, the operation $\oplus$ will always refer to the operation in a distance magma or monoid, it is called ``addition" in this paper. Then we begin by introducing the fundamental algebraic structures under investigation.

\begin{definition}[Distance magma and distance monoid]
    A distance magma $R$ is a structure in the language $\{\oplus,\leq,0\}$ with the following properties:\\
        \quad (1) (totality) $\leq$ is a total order on $R$;\\
        \quad (2) (positivity) $r\leq r\oplus s$ for all $r,s\in R$;\\
        \quad (3) (order) for all $r,s,t,u\in R$, if $r\leq t$ and $s\leq u$ then $r\oplus s\leq t\oplus u$;\\
        \quad (4) (commutativity) $r\oplus s = s\oplus r$ for all $r,s\in R$;\\
        \quad (5) (unity) $r\oplus 0 = r = 0\oplus r$ for all $r \in R$;\\
        We say that $R$ is a distance monoid if it is a distance magma and\\
        \quad (6) (associativity) $(r\oplus s)\oplus t = r\oplus (s \oplus t)$ for all $r,s,t\in R$.
\end{definition}
For positive integers $n$, we denote the number of distance monoids with $n$ non-zero elements by $DM(n)$. 

One of the most basic properties of finite distance monoids in this paper is the Archimedean complexity, which was proposed by Conant [\cite{Conant2015}, Definition 3.7.1]. 
\begin{definition}[Archimedean complexity]
            Let $R$ be a distance monoid. The Archimedean complexity of $R$, denoted by $\mathrm{arch}(R)$, is the minimum $n$ such that, for all $r_{0},r_{1}, \cdots,r_{n}\in R$, if $r_{0}\leq r_{1}\leq \cdots\leq r_{n}\in R$,
            \begin{equation*}
                r_{0}\oplus r_{1}\oplus \cdots\oplus r_{n}=r_{1}\oplus \cdots\oplus r_{n}.
            \end{equation*}
            If no such $n$ exists, we set $\mathrm{arch}(R)=\omega$.
\end{definition}
The Archimedean complexity measures how quickly repeated addition reaches a fixed point. It is a key point in our enumeration, as it allows us to classify distance monoids by their ``complexity."
    
    We now introduce the decomposition of Archimedean classes, the central technique of this paper. Firstly, we define the Archimedean classes as follows:
    \begin{definition}[Archimedean class]
            An Archimedean class of a distance magma/monoid $R$ is a maximal subset $S\subset R\setminus\{0\}$ satisfying that 
            for two non-zero elements $r,s\in S$, there exists $n\in\mathbb{N}_{>0}$ such that
            \begin{equation*}
                r\leq ns \coloneq \underbrace{s\oplus\cdots\oplus s}_{\text{$n$ times}}.
            \end{equation*}  
    \end{definition}
    We say that a finite distance magma/monoid is Archimedean if it has only one Archimedean class. Now we consider the properties of Archimedean classes.
    \begin{proposition}\label{unique}
    An Archimedean class has the following properties:
    \begin{enumerate}
        \item Let $S=\{r_{1}<\cdots<r_{k}\}$ be an Archimedean class of $R$. Then for all $r_{i},r_{j}\in S$, $r_{i}\oplus r_{j}\in S$.
        \item Let $S_{1}$ and $S_{2}$ be two different Archimedean classes of $R$. Then $S_{1}\cap S_{2}=\varnothing$.
        \item Let $S_{1}$ and $S_{2}$ be two different Archimedean classes of $R$. We have $S_{1}<S_{2}$ or $S_{2}<S_{1}$, where $S_{1}<S_{2}$ means $r<s$ for all $r\in S_1,\,s\in S_2$.
    \end{enumerate}        
    \end{proposition}
    \begin{proof}
        (1) By definition, we may assume $r_{i}\leq sr_{1}$ and $r_{j}\leq tr_{1}$. Then we have $r_{i}\oplus r_{j} \leq (s+t)r_{1}$. It implies $r_{i}\oplus r_{j}\in S$.\\
        (2) Suppose, for the sake of contradiction, that $s\in S_{1}\cap S_{2}$. For all $r\in S_{1}$, by definition, we have $r\leq ms$ for some $m\in \mathbb{N}_{>0}$. For all $t\in S_{2} $, by definition, we also have $s\leq kt$ for some $k\in \mathbb{N}_{>0}$. Notice that $r\leq ms \leq (mk)t$, it implies $S_{1}\subseteq S_{2}$. For the same reason, we also have $S_{2}\subseteq S_{1}$, which implies $S_{1}=S_{2}$. Contradiction. \\
        (3) Without loss of generality, assume that
        \[r=\min S_{1}<\min S_2=s.\]
        If there exists $u\in S_1$ such that $u\geq s$. Then we have $s\leq u\leq mr\leq ms$ for some $m\in \mathbb{N}_{>0}$. By definition, it implies $s\in S_1$, which is a contradiction. So we have $u<s$ for all $u\in S_1$, i.e., $S_1<S_2$.
    \end{proof}
    This proposition shows that Archimedean classes are well-behaved under order and addition: they are closed under addition, disjoint, and totally ordered. These properties are essential for the decomposition theorem.
    \begin{corollary}
        Adding $0$ to an Archimedean class and considering the induced order give us a distance submagma/submonoid. We call it an Archimedean submagma/submonoid, since it is Archimedean.
    
    \end{corollary}
    With the preparation work above, we can now present the decomposition of Archimedean classes.
    \begin{theorem}\label{decomp}
        Every finite distance magma/monoid has a unique decomposition of Archimedean classes $R\setminus\{0\}=\sqcup_{i} S_{i}$ where $S_i$'s are Archimedean classes of $R$ and $S_{i}<S_{j}$ if and only if $i<j$.
    \end{theorem}
    \begin{proof}
        Let $R=\{0<r_{1}<\cdots<r_{n}\}$ be a finite distance magma/monoid. We prove it by induction on $n$. We consider the arithmetic progression $r_{1},2r_{1},\cdots$. Since $R$ is finite, there exists $m\in \mathbb{N}_{>0}$ such that $mr_{1}=(m+1)r_{1}$. This implies $\ell r_1=mr_1$  for all $\ell\geq m$. We may assume $mr_1=r_{k}$. Since for all $1\leq i,j\leq k$
         \begin{equation*}
            r_{j}\leq r_{k}=mr_{1}\leq mr_{i},
        \end{equation*}
            $\{r_{1}<\cdots<r_{k}\}$ is an Archimedean class of $R$. We denote it by $S_1$.
            
            Note that $R\setminus S_1$ is still a distance magma/monoid. Then we use the hypothesis of induction to $R\setminus S_1$. It implies the existence of decomposition of Archimedean classes.
            
            Suppose that $R\setminus\{0\}=\sqcup_{i} S_{i}=\sqcup_{j}T_{j}$ are two decompositions of Archimedean classes. Then we have
            \[S_{i}=S_{i}\cap(\sqcup_{j}T_{j})=\sqcup_{j}(S_i\cap T_j).\]
            By the second statement of {Proposition \ref{unique}}, $S_i\cap T_j$ is either $\varnothing$ or $S_{i}$. So there must exist a unique $j$ such that $T_{j}=S_{i}$, which implies the uniqueness of the decomposition.
    \end{proof}
    This theorem guarantees that every finite distance magma/monoid can be uniquely decomposed into Archimedean classes, ordered by the natural order. This decomposition is the foundation for our enumeration approach, as it reduces the problem to studying simpler Archimedean submonoids.
    
    The following proposition provides us with a method to compute the number of Archimedean classes.
    \begin{proposition}
            The number of Archimedean classes of finite distance magma/monoid $R$ is $\#\{r\in R\setminus\{0\}\colon r\oplus r=r\}$.
    \end{proposition}
    \begin{proof}
    Assume $R=\{0<r_{1}<r_{2}<\cdots<r_{n}\}$. Suppose $R$ admits a decomposition
    \[r^{(1)}_{1}<\cdots<r^{(1)}_{n_{1}}<r^{(2)}_{1}<\cdots<r^{(2)}_{n_{2}}<\cdots<r^{(k)}_{1}<\cdots<r^{(k)}_{n_{k}},\]
    where $r_1^{(i)}<\cdots<r_{n_i}^{(i)},\,i=1,2,\cdots,k$ are $k$ Archimedean classes,
    with $n_{1}+\cdots+n_{k}=n$. Then we must have $r^{(i)}_{n_i}\oplus r^{(i)}_{n_i}=r^{(i)}_{n_i}$, which implies $k\leq \#\{i\colon r_{i}\oplus r_{i}=r_{i}\}$. If $k< \#\{i\colon r_{i}\oplus r_{i}=r_{i}\}$, it means that there exists an element $r^{(j)}_{m}$ with $1\leq m< n_{j}$ such that $r^{(j)}_{m}\oplus r^{(j)}_{m}=r^{(j)}_{m}$. Notice that $r_{m}^{(j)}$ and $r^{(j)}_{n_j}$ lie in distinct Archimedean classes since $\ell\cdot r^{(j)}_{m}=r^{(j)}_{m}<r^{(j)}_{n_j}$ for all $\ell\in\mathbb{N}_{>0}$. Contradiction. Therefore, we conclude that $k= \#\{i\colon r_{i}\oplus r_{i}=r_{i}\}$.
    \end{proof}
    In the end, we give an example of a finite distance monoid, which illustrates that the Archimedean complexity of a monoid is not simply the maximum complexity of its Archimedean submonoids. 
    \begin{example}\label{12567}
        Let  $R=\{0,1,2,5,6,7\}$ be a subset of $\mathbb{N}$. Define the addition as follows:
        \begin{equation*}\label{add}
            a\oplus b=\sup_{r\leq a+b,\,r\in R}r,
        \end{equation*}
        where $+$ is the natural addition in $\mathbb{N}$. Then $R$ is a finite distance monoid with complexity $3$. It has two Archimedean submonoids $\{0,1,2\}$ and $\{0,5,6,7\}$. However, these two Archimedean submonoids both have Archimedean complexity $2$.
    \end{example}
    This example leads to a deeper discussion in Section \ref{sec:arithmetic}.
    
    \section{Enumeration of $DM(n,2)$}\label{sec:enumeration}
    In this section, we focus on the enumeration of distance monoids with Archimedean complexity 2. We characterize their structure through properties of the addition operation and use this characterization to derive an exact formula for $DM(n,2)$. This formula reveals a connection to Bell numbers, allowing us to resolve Conant's conjecture.
    
     In this section, unless otherwise stated, let $R=\{0<r_{1}<r_{2}<\cdots<r_{n}\}$ be a finite distance monoid with $n$ non-zero elements and Archimedean complexity $2$. By Theorem \ref{decomp}, we can decompose $R$ into Archimedean classes 
    \[r^{(1)}_{1}<\cdots<r^{(1)}_{n_{1}}<
     r^{(2)}_{1}<\cdots<r^{(2)}_{n_{2}}<\cdots<
     r^{(k)}_{1}<\cdots<r^{(k)}_{n_{k}},\] 
     with $n_{1}+\cdots+n_{k}=n$ and $n_i>1$ for at least one $i$.

     Before giving a proof of {Theorem \ref{DM(n,2)}}, we will prove a series of lemmas. The first lemma provides us with some properties of the addition in $R$.   
     \begin{lemma}\label{lemma1}
     By convention, we denote the set $\{1,2,\cdots,n\}$ by $[n]$. The addition $\oplus$ of $R$ has the following properties: 
     \begin{enumerate}
         \item $r^{(i)}_{s}\oplus r^{(i)}_{t}=r^{(i)}_{n_{i}},\,\forall s,t\in [n_i].$
         \item Suppose $1\leq i<j\leq k$. Then $r^{(i)}_{1}\oplus r^{(j)}_{s}=r^{(i)}_{m}\oplus r^{(j)}_{s}=r^{(j)}_{t}$ for all $m\in[n_i]$  with $s\leq t\in [n_j]$.
         \item Suppose $1\leq i<j\leq k$ and $r^{(i)}_{1}\oplus r^{(j)}_{s}=r^{(j)}_{t}$ with $s\leq t\in [n_j]$. Then we have $r^{(i)}_{u}\oplus r^{(j)}_{m}=r^{(j)}_{t},\,\forall u\in[n_i],\,\forall s\leq m\leq t$ and $r\oplus r^{(j)}_{t}=r^{(j)}_{t},\,\forall r\leq r_{n_i}^{(i)}$. 
     \end{enumerate}
     \end{lemma}
     \begin{proof}
        (1) It suffices to prove $r^{(i)}_{1}\oplus r^{(i)}_{1}=r^{(i)}_{n_{i}}$. Suppose, for the sake of contradiction, that $r^{(i)}_{1}\oplus r^{(i)}_{1}<r^{(i)}_{n_{i}}$. Since $\mathrm{arch}(R)=2$, for any $n$ we have
    \begin{equation*}
        nr^{(i)}_{1}=\underbrace{r^{(i)}_{1}\oplus \cdots\oplus r^{(i)}_{1}}_{\text{$n$ times}}=r^{(i)}_{1}\oplus r^{(i)}_{1}<r^{(i)}_{n_{i}}.
    \end{equation*}
        This implies that $r^{(i)}_{1}$ and $r^{(i)}_{n_{i}}$ lie in distinct Archimedean classes, which is a contradiction.\\
        (2) Since $r^{(j)}_{s}\leq r^{(i)}_{1}\oplus r^{(j)}_{s}\leq r^{(j)}_{s}\oplus r^{(j)}_{s}=r^{(j)}_{n_j}$, we conclude $r^{(i)}_{1}\oplus r^{(j)}_{s}=r^{(j)}_{t}$ with $s\leq t\in [n_j]$.
        Notice that\[r^{(i)}_{n_i}\oplus r^{(j)}_{s}=r^{(i)}_{1}\oplus r^{(i)}_{1}\oplus r^{(j)}_{s}=r^{(i)}_{1}\oplus r^{(j)}_{s}.\] The last equality is because $R$ is of Archimedean complexity 2. So we have
        \[
            r^{(j)}_{t}=r^{(i)}_{1}\oplus r^{(j)}_{s}\leq  r^{(i)}_{m}\oplus r^{(j)}_{s}\leq r^{(i)}_{n_i}\oplus r^{(j)}_{s}=r^{(j)}_{t},
        \]
        i.e. $r^{(i)}_{m}\oplus r^{(j)}_{s}=r^{(j)}_{t},\,\forall m\in[n_i]$.\\
        (3) By the second property, we have
        \begin{align*}     
            r^{(i)}_{u}\oplus r^{(j)}_{t}&=r^{(i)}_{u}\oplus \left(r^{(i)}_{1}\oplus r^{(j)}_{s}\right)\\
            &=r^{(i)}_{u}\oplus \left(r^{(i)}_{n_i}\oplus r^{(j)}_{s}\right).
        \end{align*}
        Since $R$ is of Archimedean complexity 2, we have 
        \begin{equation*}
            r^{(i)}_{u}\oplus \left(r^{(i)}_{n_i}\oplus r^{(j)}_{s}\right)
            =r^{(i)}_{n_i}\oplus r^{(j)}_{s}=r^{(j)}_{t},
        \end{equation*}
        i.e. $ r^{(i)}_{u}\oplus r^{(j)}_{t}=r^{(j)}_{t}.$
        So we conclude
        \begin{equation*}
        r^{(j)}_{t}=r^{(i)}_{1}\oplus r^{(j)}_{s}\leq r^{(i)}_{u}\oplus r^{(j)}_{m}\leq r^{(i)}_{u}\oplus r^{(j)}_{t}=r^{(j)}_{t}.
        \end{equation*}
        By choosing $u=n_i$ and $m=t$, we have
        \[r_{n_i}^{(i)}\oplus r^{(j)}_{t}=r^{(j)}_{t}.\]
        Hence we have
        \[r^{(j)}_{t}\leq r\oplus r^{(j)}_{t}\leq r_{n_i}^{(i)}\oplus r^{(j)}_{t}=r^{(j)}_{t}\]
        for all $r\leq r_{n_i}^{(i)}$.
    \end{proof}

    The second lemma shows that the properties in Lemma \ref{lemma1} characterize finite distance monoids of Archimedean complexity 2.
    \begin{lemma}\label{lemma2}
        Let $R^{\prime}$ be a finite distance magma. Suppose $R^{\prime}$ admits a decomposition of Archimedean classes \[r^{(1)}_{1}<\cdots<r^{(1)}_{n_{1}}<
        r^{(2)}_{1}<\cdots<r^{(2)}_{n_{2}}<\cdots<
        r^{(m)}_{1}<\cdots<r^{(m)}_{n_{m}}\] 
        with $n_{1}+\cdots+n_{m}=n$ and $n_i>1$ for at least one $i$. If the addition satisfies the three properties given in Lemma \ref{lemma1}, then $R^{\prime}$ is a distance monoid of Archimedean complexity $2$.
    \end{lemma}
    \begin{proof}
        First, we prove that $R^{\prime}$ is a distance monoid. It suffices to check that the addition is associative, that is,
        \begin{equation}
    \left(r^{(i)}_{s} \oplus r^{(j)}_{t}\right) \oplus r^{(k)}_{\ell}
    = r^{(i)}_{s} \oplus \left(r^{(j)}_{t} \oplus r^{(k)}_{\ell}\right)
    \tag{$\dagger$}\label{asso}
\end{equation}
        where $s\in [n_{i}],\,t\in [n_{j}]$ and $\ell\in [n_{k}]$.

        We proceed by case analysis on the indices $i,j,k$. There are 13 cases to consider: $i=j=k,\,i=j<k,\,k<i=j,\,i=k<j,\,j<i=k,\,i<j=k,\,j=k<i,\,i<j<k,\,i<k<j,\,j<i<k,\,j<k<i,\,k<i<j,\,k<j<i$.
        \begin{enumerate}[label=Case \arabic* :, leftmargin=*]
        \item $i=j=k$.\\
            By property (1) of {Lemma \ref{lemma1}}, we have
            \begin{align*}
                &\left(r^{(i)}_{s}\oplus r^{(i)}_{t}\right)\oplus r^{(i)}_{\ell}=r^{(i)}_{n_{i}}\oplus r^{(i)}_{\ell}=r^{(i)}_{n_{i}},\\
                &r^{(i)}_{s}\oplus \left(r^{(i)}_{t}\oplus r^{(i)}_{\ell}\right)=r^{(i)}_{s}\oplus r^{(i)}_{n_{i}}=r^{(i)}_{n_{i}}.
            \end{align*}
            Thus, equation \eqref{asso} holds in this case.
        \item $i=j<k$.\\
            By properties (1) and (2) of {Lemma \ref{lemma1}}, we have
            \begin{align*}
                &\left(r^{(i)}_{s}\oplus r^{(i)}_{t}\right)\oplus r^{(k)}_{\ell}
                =r^{(i)}_{n_i}\oplus r^{(k)}_{\ell}
                =r^{(k)}_{u},
            \end{align*}
            for some $u\geq \ell$. By property (2) of {Lemma \ref{lemma1}}, we have also $r^{(i)}_{t}\oplus r^{(k)}_{\ell}=r^{(i)}_{n_i}\oplus r^{(k)}_{\ell}=r^{(k)}_{u}$. And by property (3) of {Lemma \ref{lemma1}}, $r^{(i)}_{s}\oplus r^{(k)}_{u}=r^{(k)}_{u}$, it implies \[r^{(i)}_{s}\oplus \left(r^{(i)}_{t}\oplus r^{(k)}_{\ell}\right)=r^{(k)}_{u}.\]
            Thus, equation \eqref{asso} holds in this case.
        \item $k<i=j$.\\
            By property (1) of {Lemma \ref{lemma1}}, we have
            \begin{align*}
                &\left(r^{(i)}_{s}\oplus r^{(i)}_{t}\right)\oplus r^{(k)}_{\ell}
                =r^{(i)}_{n_i}\oplus r^{(k)}_{\ell}
                =r^{(i)}_{n_i}.
            \end{align*}
            By property (2) of {Lemma \ref{lemma1}}, we have also $r^{(i)}_{t}\oplus r^{(k)}_{\ell}=r^{(i)}_{u}$ for some $u \geq t$. And by property (1) of {Lemma \ref{lemma1}}, $r^{(i)}_{s}\oplus r^{(i)}_{u}=r^{(i)}_{n_i}$. It implies
            \[r^{(i)}_{s}\oplus \left(r^{(i)}_{t}\oplus r^{(k)}_{\ell}\right)=r^{(i)}_{n_i}.\]
            Thus, equation \eqref{asso} holds in this case.
        \item $i=k<j$.\\
            By properties (2) and (3) of {Lemma \ref{lemma1}}, we have
            \begin{align*}
                &\left(r^{(i)}_{s}\oplus r^{(j)}_{t}\right)\oplus r^{(i)}_{\ell}
                =r^{(j)}_{u}\oplus r^{(i)}_{\ell}
                =r^{(j)}_{u},
            \end{align*}
            for $u\geq t$. And by properties (2) and (3) of {Lemma \ref{lemma1}}, we have
            \[r^{(i)}_{s}\oplus \left(r^{(j)}_{t}\oplus r^{(i)}_{\ell}\right)=r^{(i)}_{s}\oplus r^{(j)}_{u}=r^{(j)}_{u}.\]
            Thus, equation \eqref{asso} holds in this case.
        \item $j<i=k$.\\
            By properties (1) and (2) of {Lemma \ref{lemma1}}, we have
            \begin{align*}
                &\left(r^{(i)}_{s}\oplus r^{(j)}_{t}\right)\oplus r^{(i)}_{\ell}
                =r^{(i)}_{u}\oplus r^{(i)}_{\ell}
                =r^{(i)}_{n_i},
            \end{align*}
            where $u\geq s$. Similarly, by properties (1) and (2) of {Lemma \ref{lemma1}}, we have
            \[r^{(i)}_{s}\oplus \left(r^{(j)}_{t}\oplus r^{(i)}_{\ell}\right)=r^{(i)}_{s}\oplus r^{(i)}_{u}=r^{(i)}_{n_i}.\]
            Thus, equation \eqref{asso} holds in this case.
        \item $i<j=k$.\\
            This is equivalent to Case 3 since $R^{\prime}$ is commutative.
        \item $j=k<i$.\\
            This is equivalent to Case 2 since $R^{\prime}$ is commutative.
        \item $i<j<k$.\\
            By property (2) of {Lemma \ref{lemma1}}, we have 
            \begin{align*}
                \left(r^{(i)}_{s}\oplus r^{(j)}_{t}\right)\oplus r^{(k)}_{\ell}=r^{(j)}_{u}\oplus r^{(k)}_{\ell}=r^{(k)}_{v},
            \end{align*} 
            where $u\geq t$ and $v\geq \ell$.
            And by properties (2) and (3) of {Lemma \ref{lemma1}}, we have
            \[r^{(i)}_{s}\oplus \left(r^{(j)}_{t}\oplus r^{(k)}_{\ell}\right)=r^{(i)}_{s}\oplus r^{(k)}_{v}=r^{(k)}_{v}.\]
            Thus, equation \eqref{asso} holds in this case.
        \item $i<k<j$.\\
            By properties (2) and (3) of {Lemma \ref{lemma1}}, we have 
            \[r^{(i)}_{s}\oplus \left(r^{(j)}_{t}\oplus r^{(k)}_{\ell}\right)=r^{(i)}_{s}\oplus r^{(j)}_{v}=r^{(j)}_{v}\]
            for some $v\geq t$. By property (2) of {Lemma \ref{lemma1}}, we have
            $r^{(i)}_{s}\oplus r^{(j)}_{t}=r^{(j)}_{u}$ for some $t\leq u\leq v$. Notice that
            \[r^{(j)}_{u}\oplus r^{(k)}_{\ell}\geq r^{(j)}_{t}\oplus r^{(k)}_{\ell}= r^{(j)}_{v}.\]   
            And by property (3) of {Lemma \ref{lemma1}}, we have also
            \[r^{(j)}_{u}\oplus r^{(k)}_{\ell}\leq r^{(j)}_{v}\oplus r^{(k)}_{\ell}=r^{(j)}_{v}.\]
            So we have $r^{(j)}_{u}\oplus r^{(k)}_{\ell}=r^{(j)}_{v}$. It implies
            \begin{align*}
                \left(r^{(i)}_{s}\oplus r^{(j)}_{t}\right)\oplus r^{(k)}_{\ell}=r^{(j)}_{u}\oplus r^{(k)}_{\ell}=r^{(j)}_{v}.
            \end{align*} 
            Thus, equation \eqref{asso} holds in this case.
        \item $j<i<k$.\\
            By property (2) of {Lemma \ref{lemma1}}, we have 
            \begin{align*}
                \left(r^{(i)}_{s}\oplus r^{(j)}_{t}\right)\oplus r^{(k)}_{\ell}=r^{(i)}_{u}\oplus r^{(k)}_{\ell}=r^{(k)}_{v},
            \end{align*} 
            where $u\geq s$ and $v\geq \ell$. By property (2) of {Lemma \ref{lemma1}}, we have
            $r^{(j)}_{t}\oplus r^{(k)}_{\ell}=r^{(k)}_{w}$ for some $\ell\leq w\leq v$. Notice that by property (2) of {Lemma \ref{lemma1}}
            \[r^{(i)}_{s}\oplus r^{(k)}_{w}\geq r^{(i)}_{1}\oplus r^{(k)}_{\ell}= r^{(i)}_{u}\oplus r^{(k)}_{\ell}=r^{(k)}_{v}.\]   
            And by property (3) of {Lemma \ref{lemma1}}, we have also
            \[r^{(i)}_{s}\oplus r^{(k)}_{w}\leq r^{(i)}_{u}\oplus r^{(k)}_{v}=r^{(k)}_{v}.\]
            So we have $r^{(i)}_{s}\oplus r^{(k)}_{w}=r^{(k)}_{v}$. It implies
            \[r^{(i)}_{s}\oplus \left(r^{(j)}_{t}\oplus r^{(k)}_{\ell}\right)=r^{(i)}_{s}\oplus r^{(k)}_{w}=r^{(k)}_{v}.\]
            Thus, equation \eqref{asso} holds in this case.
        \item $j<k<i$.\\
            This is equivalent to Case 9 since $R^{\prime}$ is commutative.
        \item $k<i<j$.\\
            This is equivalent to Case 10 since $R^{\prime}$ is commutative.
        \item $k<j<i$.\\
            This is equivalent to Case 8 since $R^{\prime}$ is commutative.
        \end{enumerate}

        Then we prove that for all $r^{(i)}_{s}\leq r^{(j)}_{t}\leq r^{(k)}_{\ell}$ with $s\in [n_{i}],\,t\in [n_{j}]$ and $\ell\in [n_{k}]$, the following equation holds:
            \begin{equation}\label{arcom2}
                r^{(i)}_{s}\oplus r^{(j)}_{t}\oplus r^{(k)}_{\ell}=r^{(j)}_{t}\oplus r^{(k)}_{\ell}.
            \end{equation}
            There are only 4 cases: $i=j=k,\,i<j=k,\,i=j<k$ and $i<j<k$.
        \begin{enumerate}[label=Case \arabic* :, leftmargin=*]
        \item $i=j=k$\\
            By property (1) of {Lemma \ref{lemma1}}, we have
            \[
                r^{(i)}_{s}\oplus r^{(i)}_{t}\oplus r^{(i)}_{\ell}=r^{(i)}_{n_{i}}\oplus r^{(i)}_{\ell}=r^{(i)}_{n_{i}}=r^{(i)}_{t}\oplus r^{(i)}_{\ell}.
            \]

        \item $i<j=k$\\
            By properties (1) and (3) of {Lemma \ref{lemma1}}, we have
            \[
                r^{(i)}_{s}\oplus r^{(j)}_{t}\oplus r^{(j)}_{\ell}=r^{(i)}_{s}\oplus r^{(j)}_{n_{j}}=r^{(j)}_{n_{j}}=r^{(j)}_{t}\oplus r^{(j)}_{\ell}.
            \]
        \item $i=j<k$\\
            By properties (1) and (2) of {Lemma \ref{lemma1}}, we have
            \[
                r^{(i)}_{s}\oplus r^{(i)}_{t}\oplus r^{(k)}_{\ell}=r^{(i)}_{n_{i}}\oplus r^{(k)}_{\ell}=r^{(i)}_{t}\oplus r^{(k)}_{\ell}.
            \]
        \item $i<j<k$\\
            By property (2) of {Lemma \ref{lemma1}}, we can suppose $r^{(j)}_{t}\oplus r^{(k)}_{\ell}=r^{(k)}_{u}$ with $\ell\leq u\leq n_k$. Then by property (3) of {Lemma \ref{lemma1}}, we have
            \[
                r^{(i)}_{s}\oplus r^{(j)}_{t}\oplus r^{(k)}_{\ell}=r^{(i)}_{s}\oplus r^{(k)}_{u}=r^{(k)}_{u}=r^{(j)}_{t}\oplus r^{(k)}_{\ell}.
            \]
        \end{enumerate}
            
        This completes the verification of equation \eqref{arcom2}. Consequently, we have $\mathrm{arch}(R^{\prime})\leq 2$. Moreover, since at least one $n_i$ is greater than 1, it follows that $\mathrm{arch}(R^{\prime})=2$. 
    \end{proof}
    By establishing the necessary and sufficient conditions for a distance magma to be a distance monoid of Archimedean complexity 2, Lemmas \ref{lemma1} and \ref{lemma2} allow us to bypass the direct verification of the associative law of addition in the enumeration. The next lemma provides the tools of enumeration.
    \begin{lemma}\label{lemma3}
    We have the following counting results: 
    \begin{enumerate}
        \item For $n\geq 1$, we define \[A(n)=\{(a_i)_{i=1}^n\in [n]^n\colon a_i\geq i,\,\forall i\in [n] \text{ and } a_m=a_i,\,\forall i\leq m\leq a_i\},\] 
        where $[n]$ denotes the set $\{1,2,\cdots,n\}$. Then $|A(n)|=2^{n-1}$.
        \item We define a partial order $\leq $ on $A(n)$ as below: 
        \[\forall a,a^{\prime}\in A(n),\,a\leq  a^{\prime}\,\iff \,\,a_i\leq a^{\prime}_i,\,\forall i\in [n].\]
        For $k\geq 1,$ we define \[A_{k}(n)=\left\{(a^{(i)})_{i=1}^k\in (A(n))^k\colon a^{(k)}\leq \cdots \leq a^{(1)}\right\}.\] Then $|A_{k}(n)|=(k+1)^{n-1}$.
    \end{enumerate}
    \end{lemma}
    \begin{proof}
        (1) We have a bijection between $A(n)$ and the subsets of $[n]$ containing $n$ as follows:  
        \[a=(a_i)_{i=1}^n \longleftrightarrow \{i\in [n]\colon a_i=i\}.\]
        Thus, $|A(n)|=2^{n-1}.$\\
        (2) Considering the partial order $\subset$ on the subsets of $[n]$, we have a bijection between $A_{k}(n)$ and a $k$-chain of subsets of $[n]$ containing $n$: 
        \[a^{(k)}\leq \cdots \leq a^{(1)} \longleftrightarrow \{S_{1}\subset \cdots \subset S_{k}\},\]
        where $S_{i}=\left\{j\in [n]\colon a^{(i)}_j=j\right\}$ for all $i\in [k]$. Then we can see
        \begin{align*}
            |A_{k}(n)|&=\sum_{n_{1}+\cdots+n_{k+1}=n-1,n_{i}\in \mathbb{N}}
            \binom{n-1}{n_1}\binom{n-1-n_1}{n_2}\cdots\binom{n-1-\sum_{i=1}^{k-1}n_{i}}{n_{k}}\\
            &=\sum_{n_{1}+\cdots+n_{k+1}=n-1,n_{i}\in \mathbb{N}}\frac{(n-1)!}{n_1!\cdots n_{k+1}!}=(k+1)^{n-1}. \qedhere
        \end{align*}
    \end{proof}
    With the preparation work above, now we can give a proof of Theorem \ref{DM(n,2)} by establishing a one-to-one correspondence.
    \begin{proof}[Proof of {Theorem \ref{DM(n,2)}}]
        By {Lemma \ref{lemma2}}, a distance magma $R$ is a distance monoid with complexity $2$ if and only if it has a decomposition of Archimedean classes such that one class has at least two elements and its addition satisfies the three properties in {Lemma \ref{lemma1}}. Furthermore, the structure of $R$ is completely determined by its addition operation. Therefore, it suffices to count the number of admissible addition operations.
        
        Since the addition within each class $\big\{r^{(i)}_{1},\cdots,r^{(i)}_{n_i}\big\}$ is determined by property (1) of {Lemma \ref{lemma1}}, we only need to consider the addition between elements from different classes i.e. between
        $\big\{r^{(i)}_{1},\cdots,r^{(i)}_{n_i}\big\}$ and $\big\{r^{(j)}_{1},\cdots,r^{(j)}_{n_j}\big\}$ for $i\neq j$. By property (2) of {Lemma \ref{lemma1}}, we only need to consider $r^{i}_{n_{i}}\oplus r^{j}_{1},\cdots, r^{(i)}_{n_{i}}\oplus r^{(j)}_{n_j}$ for $i\in [j-1]$. 
        Assume $a^{ij}_{k}\in [n]$ is the index of $r^{(i)}_{n_{i}}\oplus r^{(j)}_{k}$. It satisfies
        \begin{itemize}
            \item $a^{ij}_k\geq k$ for $k\in [n_j]$,
            \item $a^{ij}_k=a_{m}^{ij}$ if $k\leq m \leq a^{ij}_k$.
        \end{itemize}
        
        Following the notation of {Lemma \ref{lemma3}}, we can see\[\left(a^{ij}_k\right)_{1\leq k\leq n_j}\in A(n_j) \text{ and }\left(\left(a^{ij}_k\right)_{1\leq k\leq n_j}\right)_{1\leq i\leq j-1}\in A_{j-1}(n_j).\]
        Finally, by {Lemma \ref{lemma3}}, we obtain
        \[DM(n,2)=\sum_{k=1}^{n-1}\sum\limits_{\substack{n_{1}+\cdots+n_{k}=n \\ n_{i}\in \mathbb{N}_{>0}}}\prod\limits_{j=1}^k j^{n_j-1}. \qedhere \] 
    \end{proof} 
    \begin{remark}
        The most difficult part of computing $DM(n)$ or $DM(n,k)$ lies in verifying the associative law of addition. Lemma \ref{lemma1} and Lemma \ref{lemma2} allow us to circumvent this fundamental difficulty, although the method does not generalize to arbitrary cases.
    \end{remark}

    This theorem thus confirms a conjecture of Conant concerning the order of $DM(n)$.
    \begin{corollary}[{[\cite{Conant2015}, Conjecture 5.1.8(b)]}]
        \[DM(n)=\Omega(b^n), \,\forall b>0,\] where $f(n)=\Omega(g(n))$ means that there are $c\in \mathbb{R}_{>0}$ and $n_0\in \mathbb{N}_{>0}$ such that $f(n)\geq cg(n)$ for all $n>n_0$.
    \end{corollary}
    \begin{proof}
        Actually, {Theorem \ref{DM(n,2)}} implies $DM(n,2)=B_n - 1$, where $B_n$ is the Bell number, i.e. the number of partitions of a set of $n$ elements. Notice that for a fixed $k$, we have
    \begin{align*}
        \prod_{j=1}^{k}\frac{1}{1-jx}= \sum_{m=0}^{\infty}\left(\sum_{\substack{m_{1}+\cdots+m_{k}=m \\ m_{i}\in \mathbb{N}}}
        \prod_{j=1}^{k}j^{m_j}\right)x^m.
    \end{align*}
    It implies that $DM(n,2)+1$ is the coefficient of $x^n$ of $\sum_{k=1}^{\infty}x^k\prod_{j=1}^{k}\frac{1}{1-jx}$. 
    Notice that $x^k\prod_{j=1}^{k}\frac{1}{1-jx}$ is also the generating function of the second kind of Stirling number, i.e.
    \[x^k\prod_{j=1}^{k}\frac{1}{1-jx}=\sum_{n=k}^{\infty}S(n,k)x^n.\]
    It yields
    \begin{align*}
        \sum_{k=1}^{\infty}x^k\prod_{j=1}^{k}\frac{1}{1-jx}
        &=\sum_{k=1}^{\infty}\sum_{n=k}^{\infty}S(n,k)x^n\\
        &=\sum_{n=1}^{\infty}\left(\sum_{k=1}^{n}S(n,k)\right)x^n\\
        &=\sum_{n=1}^{\infty}B_nx^n.
    \end{align*}
    So we have $DM(n,2)=B_n - 1$. 
    
    de Bruijn [\cite{de2014asymptotic}] derived the following asymptotic estimate for the Bell number $B_n$: 
    \[\frac{\ln B_n}{n}=\ln n-\ln\ln n-1+\frac{\ln\ln n}{\ln n}+\frac{1}{\ln n}+\frac{1}{2}\left(\frac{\ln\ln n}{\ln n}\right)^2+O\left(\frac{\ln\ln n}{(\ln n)^2}\right).\]
    In particular, we conclude that for every $\varepsilon>0$ there exists $n_0 = n_0(\varepsilon)$ such that, for all $n > n_0$,
    \[\left(\frac{n}{e\ln n}\right)^n<B_n<\left(\frac{n}{e^{1-\varepsilon}\ln n}\right)^n.\]
    It implies $DM(n)\geq DM(n,2)=B_n - 1=\Omega(b^n), \,\forall b>0$.
    \end{proof}
    Motivated by the equality $DM(n,2) = B_n - 1$, we are led to the following conjecture.
    \begin{conjecture}
        There exists a natural bijection between the set partitions of $[n]$ and the distance monoids with $n$ non-zero elements that have Archimedean complexity $1$ or $2$.
    \end{conjecture}

    \section{Arithmetic progression and Archimedean complexity}\label{sec:arithmetic}
    In this section, we explore the relationship between Archimedean complexity and the existence of long arithmetic progressions. We show via a counterexample that high complexity does not necessarily imply long progressions. However, under a size condition, we prove that progressions of length equal to the complexity must exist.
    
    {Example} \ref{12567} illustrates a key point: a finite distance monoid of Archimedean complexity $k$ can be decomposed into Archimedean submonoids of Archimedean complexity strictly less than $k$. By applying [\cite{Conant2015}, Proposition 3.7.19] to an Archimedean submonoid, we see that the Archimedean complexity coincides exactly with the maximum length of an arithmetic progression within it. This leads to the following question.
    \begin{question}
        Let $R$ be a finite distance monoid of Archimedean complexity $k$. Is there a connection between $\mathrm{arch}(R)$ and the maximum length of an arithmetic progression (with distinct elements) i.e. $\{r< 2r<\cdots<kr\}\subseteq R$?
    \end{question}
    The answer is no since we have the following example.
    \begin{example}
    We construct a sequence of finite distance monoids $R_n$ by induction as follows. 
    Define $R_2=\{0,1,2\}$. Given $R_{n}=\{0<r_{1}<r_{2}<\cdots<r_{n}\}$,  define $R_{n+1}$ by appending new elements: $R_{n+1}=R_n\cup\{2r_{n}+1<2r_{n}+1+r_{1}<\cdots<2r_{n}+1+r_{n}\}$. The addition is defined in the natural way as in \eqref{add}, and one can verify by induction that it is associative. By construction, we have $\mathrm{arch}(R_n)=n$. However, the maximum length of an arithmetic progression in $R_n$ is only $2$, because every Archimedean submonoid of $R_n$ has Archimedean complexity at most $2$.
    \end{example}
    However, under the stronger condition $|R| > 1 + 4 \cdot \mathrm{arch}(R)$, we obtain the following theorem, which guarantees the existence of arithmetic progression of length $\mathrm{arch}(R)$.
    \begin{theorem}\label{arith pro}
            Fix $k\in \mathbb{N}_{>0}$. If $n>4k$, every distance monoid with $n$ non-zero elements and Archimedean complexity $n-k$ has an arithmetic progression $\{r<2r<\cdots<(n-k)r\}.$
    \end{theorem}
    \begin{proof}
    Let $S$ be a distance monoid with $n$ non-zero elements and $\mathrm{arch}(S)=n-k$. By the definition of Archimedean complexity, there exists a sequence 
    $r_0\leq r_1\leq \cdots \leq r_{n-k-1}\in S,$ such that 
    $$r_1\oplus\cdots \oplus r_{n-k-1}<r_0\oplus r_1\oplus\cdots \oplus r_{n-k-1}.$$
    Applying [\cite{Conant2015}, Proposition 3.7.21] to this sequence yields a chain of inequalities:
    \begin{equation}\label{eq arith}
        r_{n-k-1}<r_{n-k-1}\oplus r_{n-k-2}<\cdots <r_{n-k-1}\oplus r_{n-k-2}\oplus\cdots \oplus r_0.
    \end{equation} 
    And these $n-k$ elements are in the same Archimedean class, we denote it by $S_0$. And we denote the set of elements smaller than $r_{n-k-1}$ by $S_1$, the set of elements bigger than $r_{n-k-1}\oplus r_{n-k-2}\oplus\cdots \oplus r_0$ by $S_2$.

    Since $r_{n-k-1}\in S_0,$ we have $r_i\notin S_2,\,\forall 0\leq i\leq n-k-1$. If $r_0 \in S_0$, there exists $m\in \mathbb{N}_{>0}$ such that 
    $r_{n-k-1}\oplus r_{n-k-2}\oplus \cdots \oplus r_0\leq mr_0$. Since $r_{n-k-1}\oplus r_{n-k-2}\oplus\cdots \oplus r_0>r_{n-k-1}\oplus r_{n-k-2}\oplus \cdots \oplus r_1\geq (n-k-1)r_0$, we can see $m\geq n-k$. Then we find an arithmetic progression $\{r_0<2r_0<\cdots<(n-k)r_0\}$.

    So we may assume that $r_0,r_1,\cdots ,r_j\in S_1,\,0\leq j \leq n-k-2.$
    By \eqref{eq arith}, we have these $j+1$ different elements
    $$r_j<r_j\oplus r_{j-1}<\cdots <r_j\oplus r_{j-1}\oplus\cdots \oplus r_0\in S_1.$$
    Therefore, we have $j\leq k-1.$

    Assume $r=r_j\oplus r_{j-1}\oplus\cdots \oplus r_0\in S_1$. Then we have
    \begin{align*}
        r_{n-k-1}&<\cdots <r_{n-k-1}\oplus r_{n-k-2}\oplus\cdots \oplus r_{j+1}\\
        &<r_{n-k-1}\oplus r_{n-k-2}\oplus \cdots \oplus r_{j+1}\oplus r.
    \end{align*}

    \begin{claim}\label{claim}
        We have the following chain of inequalities:
        \begin{align*}
            r_{j+1}&<r\oplus r_{j+1}<r_{j+1}\oplus r_{j+2}<r\oplus r_{j+1}\oplus r_{j+2}\\
            &<\cdots<r\oplus r_{j+1}\oplus r_{j+2}\oplus\cdots \oplus r_{n-k-2}.
        \end{align*}
         
    \end{claim}

    \begin{claimproof}
    By \eqref{eq arith}, we only need to prove that
    $$r\oplus r_{j+1}\oplus\cdots \oplus r_m<r_{j+1}\oplus\cdots \oplus r_m\oplus r_{m+1},\,j+1\leq m\leq n-k-3.$$
    If not, we may have
    $$r\oplus r_{j+1}\oplus\cdots \oplus r_m=r_{j+1}\oplus\cdots \oplus r_m\oplus r_{m+1}$$ for some $m$. Since
    $$r\oplus r_{j+1}\oplus\cdots \oplus r_m\leq r_{j+1}\oplus r_{j+1}\oplus \cdots \oplus r_m\leq r_{j+1}\oplus \cdots \oplus r_m\oplus r_{m+1},$$ it implies
    \[r\oplus r_{j+1}\oplus\cdots \oplus r_m= r_{j+1}\oplus r_{j+1}\oplus \cdots \oplus r_m.\]
    So we obtain
    \begin{align*}
        &r_{j+1}\oplus r_{j+1}\oplus r_{j+1}\oplus\cdots \oplus r_m\\
        =\ &r_{j+1}\oplus (r_{j+1}\oplus r_{j+1}\oplus\cdots \oplus r_m)\\
        =\ &r_{j+1}\oplus (r\oplus r_{j+1}\oplus \cdots \oplus r_m)\\
        =\ &r\oplus (r_{j+1}\oplus r_{j+1}\oplus \cdots \oplus r_m)\\
        =\ &r\oplus (r\oplus r_{j+1}\oplus \cdots \oplus r_m)\\
        =\ &(r\oplus r)\oplus r_{j+1}\oplus \cdots \oplus r_m\\
        \leq\ & r_{j+1}\oplus r_{j+1}\oplus \cdots \oplus r_m,
    \end{align*}
    which means $$r_{j+1}\oplus \cdots \oplus r_m\oplus r_{j+1}=r_{j+1}\oplus\cdots \oplus r_m\oplus \ell\cdot r_{j+1}, \forall\ell\in\mathbb{N}_{>0}.$$
    
    Since $r_{j+1}, r_{m+1}, r_{m+2}$ are in the same Archimedean class, it yields
    \begin{align*}
        r_{j+1}\oplus \cdots \oplus r_{m+2}
        &\leq r_{j+1}\oplus\cdots \oplus r_m\oplus\ell\cdot r_{j+1}\\
        &=r_{j+1}\oplus\cdots \oplus r_m\oplus r_{j+1}\\
        &\leq r_{j+2}\oplus \cdots \oplus r_{m+2},
    \end{align*}
    which contradicts \eqref{eq arith}.
    \end{claimproof}
    By this claim, we have at least $2(n-k-j-1)\geq 2(n-2k)>n$ distinct elements in $S_0$, which is a contradiction.
\end{proof}

\section{$DM(n,n-k)$ with $n\gg k$}\label{sec:asymptotic}
    In this section, we study the asymptotic behavior of $DM(n,n-k)$ for fixed $k$
and large $n$. We establish a lower bound via explicit constructions and an upper bound by analyzing the structure of such monoids. We also provide an exact formula for $DM(n,n-2)$ in the end.

    We begin by constructing a lower bound for $DM(n,n-k)$ with $n\geq k+2$.
\begin{proposition}\label{lower bound for n-k}
    Fix $k\in \mathbb{N}_{>0}$. For $n\geq k+2$, 
    \[DM(n,n-k)\geq \binom{n-2}{k}.\]
\end{proposition}
\begin{proof}
    We construct a family of distance monoids as follows. Let the base set be \[S=\{0,1,2,\cdots,n-k\}\cup\{x_{1},\cdots,x_{k}\},\] where $x_{j}=i_{j}+{n^{-(k+1-j)}}$ for some $i_{j}\in [n-k-1]$ with $i_{1}\leq i_{2}\leq \cdots \leq i_{k}$. There are $\binom{n-2}{k}$ choices for $k$-tuple $(x_{1},\cdots,x_{k})$. Define the addition by \[a\oplus b=\sup_{s\leq a+b,\,s\in S}s,\]where $+$ is the natural addition in $\mathbb{R}$. To verify that $S$ is a finite distance monoid with $n$ non-zero elements and $\mathrm{arch}(S)=n-k$, it suffices to check that the addition is associative. 
    
    For all $1\leq i,j\leq k$ and $m\in \{0\}\cup [n-k]$, we have
    \begin{align*}
        &x_{i}\oplus x_{j}=\min\{\lfloor x_i+x_j\rfloor,n-k\},\\
        &m\oplus x_{i}=\min\{\lfloor m+x_i\rfloor,n-k\}.
    \end{align*}
    It yields
    \[(a\oplus b)\oplus c=a\oplus (b\oplus c)=\min\{\lfloor a+b+c\rfloor,n-k\},\,\forall a,b,c\in S.\]
    which means these $\binom{n-2}{k}$ distance monoids are exactly what we want.
\end{proof}

    Motivated by this proposition and $DM(n,n-1)=2n-2$, we propose Theorem \ref{DM(n,n-k)}. With the preparation work in Section \ref{sec:arithmetic}, we can give a proof of Theorem \ref{DM(n,n-k)}.
\begin{proof}[Proof of Theorem \ref{DM(n,n-k)}]
    The case $k=0$ is trivial since Conant proves that $DM(n,n)=1$ in [\cite{Conant2015}]. For $k\in \mathbb{N}_{>0}$ and sufficiently large $n$, let $S$ be a distance monoid with $n$ non-zero elements and Archimedean complexity $n-k$. By {Theorem \ref{arith pro}}, there exists an element $r_0\in S$, such that the arithmetic progression $\{r_0<2r_0<\cdots<(n-k)r_0\}$ is contained in $S$. 
    
    \begin{claim}\label{minmax}
        Suppose $r$ the smallest element of the Archimedean class containing $r_0$. Then $\{r<2r<\cdots<(n-k)r\}$ is an arithmetic progression in $S$. In addition, $(n-k)r$ is the largest element in this Archimedean class.
    \end{claim}
    \begin{claimproof}
        We can suppose that $r\leq r_0\leq mr$. So $(n-k)r\leq (n-k)r_0\leq m(n-k)r$. But $\mathrm{arch}(S)=n-k$, it implies $m(n-k)r=(n-k)r$. So $(n-k)r=(n-k)r_0$. And $(n-k-1)r\leq (n-k-1)r_0<(n-k)r_0\leq (n-k)r$. So $\{r<2r<\cdots<(n-k)r\}\subseteq S$ is an arithmetic progression. If $s$ is an element in this Archimedean class, we have $s\leq m^{\prime}r\leq (n-k)r$ since $\mathrm{arch}(S)=n-k$.
    \end{claimproof}
    We partition $S\setminus\{0\}$ into three parts relative to this progression:
    \begin{align*}
        &S_0=\{s\in S\colon r\leq s\leq (n-k)r\},\\
        &S_1=\{s\in S\colon s< r\},\\
        &S_2=\{s\in S\colon s> (n-k)r\}.
    \end{align*}
    Moreover, we can write them as follows
    \begin{align*}
        &S_0=\{r,2r,\cdots,(n-k)r\}\sqcup\{a_1,a_2,\cdots,a_m\},\\
        &S_1=\{b_1,b_2,\cdots,b_p\},\\
        &S_2=\{c_1,c_2,\cdots,c_q\}.
    \end{align*}
    where $b_1<\cdots<b_p<r<a_1<\cdots<a_m<(n-k)r<c_1<\cdots<c_q$ with $p+q+m=k$.
    \begin{claim}\label{claim 5.2}
        $S_1\sqcup\{0\},\,S_2\sqcup\{0\}$ and $S_0\sqcup\{0\}$ are three distance monoids.
    \end{claim}
    \begin{claimproof}
        By definition, $c_{i}\oplus c_j\in S_2$ for all $1\leq i,j\leq q$. So $S_2\sqcup\{0\}$ is a distance monoid. 

        By definition, for all $x,y\in S_0$, $r\leq x\leq (n-k)r,\,r\leq y\leq (n-k)r$. Therefore we have $2r\leq x\oplus y\leq (n-k)r$, i.e. $x\oplus y\in S_0$. So $S_0\sqcup\{0\}$ is a distance monoid.
        
        If $b_{i}\oplus b_j\geq r$ for some $1\leq i<j\leq p$, it implies $r\leq b_{i}\oplus b_j\leq 2b_j$. So $b_j$ and $r$ are in the same Archimedean class. Contradiction. We can see that $S_1\sqcup\{0\}$ is a distance monoid. 
    \end{claimproof}
    Assume $N=\sum_{m=1}^{k}DM(m)$. Since $S_1,S_2$ are two distance monoids with at most $k$ non-zero elements, there are at most $N^2$ choices of $S_1,S_2.$

    There are $\binom{n-k-2+m}{m}=O(n^m)$ choices of $\{\ell_i\}_{i=1}^m$, where $\ell_ir<a_i<(\ell_i+1)r$.

    Now we only need to consider the addition in $S_0$ and the addition between $S_0,S_1$ and $S_2$. 
    
    To consider the addition in $S_0$, we only need to consider
    $a_i\oplus\ell r$ and $a_i\oplus a_j$, where $\ell \in [n-k],\, 1\leq i \leq j \leq m$. Since the possible values are all like $\ell r$ or $a_t$. It suffices to consider $a_i\oplus r$ and $a_i\oplus a_j$, where $1\leq i \leq j \leq m$. Notice that $a_{i}\oplus r$ must be $(\ell_i+1)r$ or $(\ell_i+2)r$ or $a_{t}$ with $i\leq t\leq m$ and $a_i\oplus a_j$ must be $(\ell_i+\ell_j)r$ or $(\ell_i+\ell_j+1)r$ or $(\ell_i+\ell_j+2)r$ or $a_{t}$ with $j\leq t\leq m$. Hence the number of the addition in $S_0$ is at most 
    \[\prod_{x=1}^{m}(m-x+3)\cdot \prod_{1\leq i\leq j\leq m }(m-j+4)\leq (k+3)^{\frac{k(k+3)}{2}}.\]
    To consider the addition between $S_0,S_1, S_2,$ we only need to consider: 
\begin{enumerate}
    \item the value of $b_i\oplus r,$ where $1\leq i \leq p$. It must be $r$ or $2r$ or $a_t,$ where  $1\leq t \leq m.$
    \item the value of $b_i\oplus a_j,$ where $1\leq i \leq p,\,1\leq j \leq m$. It must be $(\ell_j+1)r$ or $(\ell_j+2)r$ or $a_t,$ where $j\leq t \leq m$.
    \item the value of $c_i\oplus r$, where $1\leq i \leq q$. It must be $c_t,$ where $i\leq t \leq q.$
    \item the value of $c_i\oplus a_j,$ where $1\leq i \leq q,\,1\leq j \leq s.$ It must be $c_t,$ where $i\leq t \leq q.$
    \item the value of $b_i\oplus c_j,$ where $1\leq i \leq p,\,1\leq j \leq q.$
    It must be $c_t,$ where $j\leq t \leq q.$
\end{enumerate}
For each of them, there are at most $(k+2)$ possibilities. Hence the number of the addition between $S_0,S_1$ and $S_2$ is at most $(k+2)^{3k^2+2k}$.

Multiplying all the possibilities that we have obtained above together, we finally obtain
\begin{align*}
&DM(n,n-k)\\
\leq &\sum_{m=0}^{k}(k-m+1)N^2\binom{n-k-2+m}{m}(k+3)^{\frac{k(k+3)}{2}} (k+2)^{3k^2+2k}\\
\leq & \sum_{m=0}^{k}(k+1)N^2(k+3)^{\frac{7k^2+7k}{2}}\binom{n-2}{m}=O(n^k). \qedhere
\end{align*}
\end{proof}
\begin{corollary}
    Fix $k\in \mathbb{N}$, there exist two constants $C_1(k)$ and $C_2(k)$ such that for all sufficiently large $n$,
    \[C_1(k)\leq \frac{DM(n,n-k)}{n^k}\leq C_2(k).\]
\end{corollary}
\begin{proof}
    Combining Proposition \ref{lower bound for n-k} and Theorem \ref{DM(n,n-k)}, it is trivial since we can choose 
    \[C_1(k)=\frac{1}{k!}-\varepsilon,\ C_2(k)=(1+\varepsilon)\cdot\frac{N(k)^2(k+1)(k+3)^{\frac{7k^2+7k}{2}}}{k!},\]
    where $N(k)=\sum_{m=1}^{k}DM(m)$ and $\varepsilon>0$.
\end{proof}
\begin{remark}
    By applying the same method in Theorem \ref{DM(n,n-k)} and carrying out a careful case analysis, we can determine the exact value of $DM(n,n-2)$ for all $n\geq 9$ by the formula
    \begin{align*}
        DM(n,n-2)&=2n^2-2n-8+\delta_{3}(n)+\delta_{3}(n+1),
    \end{align*}
    where $\delta_{3}(n)=1$ if $3\mid n$ and $\delta_{3}(n)=0$ otherwise.
\end{remark}
\printbibliography[title={Bibliography}]
\end{document}